\documentclass[11pt]{article}
\usepackage{latexsym,amssymb,float,amsmath,amsthm,enumerate,cite,geometry,tikz}
\newtheorem{theorem}{Theorem}

\usepackage{setspace}
\allowdisplaybreaks

\usetikzlibrary{decorations.text}
\usetikzlibrary{decorations.pathmorphing}
\usepackage{pgfplots}

\begin{document}
\onehalfspace
\title{Degree Deviation and Spectral Radius}
\author{Dieter Rautenbach \and Florian Werner}
\date{}
\maketitle
\vspace{-1cm}
\begin{center}
Institute of Optimization and Operations Research, Ulm University, Ulm, Germany\\
\texttt{$\{$dieter.rautenbach,florian.werner$\}$@uni-ulm.de}
\end{center}
\begin{abstract}
For a finite, simple, and undirected graph $G$ 
with $n$ vertices, $m$ edges, and largest eigenvalue $\lambda$,
Nikiforov introduced the degree deviation of $G$ as
$s=\sum_{u\in V(G)}\left|d_G(u)-\frac{2m}{n}\right|$.
Contributing to a conjecture of Nikiforov, we show
$\lambda-\frac{2m}{n}\leq \sqrt{\frac{2s}{3}}$.
For our result, 
we show that the largest eigenvalue of a graph 
that arises from a bipartite graph with $m_{A,B}$ edges
by adding $m_A$ edges within one of the two partite sets is at most 
$\sqrt{m_A+m_{A,B}+\sqrt{m_A^2+2m_Am_{A,B}}}$,
which is a common generalization of results 
due to Stanley and Bhattacharya, Friedland, and Peled.\\[3mm]
\end{abstract}

\section{Introduction}\label{sec1}

We consider finite, simple, and undirected graphs and use standard notation and terminology.
For a graph $G$ with $n$ vertices and $m$ edges,
Nikiforov \cite{ni} introduced the {\it degree deviation} $s(G)$ of $G$ as
$s(G)=\sum\limits_{u\in V(G)}\left|d_G(u)-\frac{2m}{n}\right|$.
For the {\it spectral radius} $\lambda(G)$ of $G$,
which is the largest eigenvalue of the adjacency matrix of $G$,
he showed that
$\lambda(G)-\frac{2m}{n}\leq \sqrt{s(G)}$
and conjectured 
$\lambda(G)-\frac{2m}{n}\leq \sqrt{\frac{s(G)}{2}}$
for sufficiently large $n$ and $m$.
Zhang \cite{zh} showed 
$\lambda(G)-\frac{2m}{n}\leq \sqrt{\frac{9s(G)}{10}}$.

We make further progress on Nikiforov's conjecture by showing the following.

\begin{theorem}\label{theorem1}
If $G$ is a graph with $n$ vertices and $m$ edges, then 
$$\lambda(G)-\frac{2m}{n}\leq \sqrt{\frac{2s(G)}{3}}.$$
\end{theorem}
For the proof of Theorem \ref{theorem1}, 
we establish a new bound on the spectral radius of a graph,
which is a common generalization of results 
due to Stanley \cite{st} and Bhattacharya et al.~\cite{bhfrpe}.
For a graph $G$ with $n$ vertices and $m$ edges,
Stanley \cite{st} showed $\lambda(G)\leq \sqrt{2m}$;
in fact, he showed a slightly stronger bound.
Provided that $G$ is bipartite,
Bhattacharya et al.~\cite{bhfrpe} showed $\lambda(G)\leq \sqrt{m}$,
which had been shown before by Nosal \cite{no} for triangle-free graphs.

\begin{theorem}\label{theorem2}
Let $G$ be a graph whose vertex set is partitioned into the two sets $A$ and $B$.
If the edge set of $G$ consist of 
$m_A$ edges with both endpoints in $A$ and 
$m_{A,B}$ edges with one endpoint in $A$ and one endpoint in $B$,
then 
$$\lambda(G)\leq \sqrt{m_A+m_{A,B}+\sqrt{m_A^2+2m_Am_{A,B}}}.$$
\end{theorem}
Complete split graphs $CS(q,n)$ with $q$ universal vertices and $n-q$ vertices of degree $q$ 
show that Theorem \ref{theorem2} is essentially best possible.
In fact, it is known \cite{hoshfa} that the spectral radius of 
$CS(q,n)$ is $\frac{1}{2}\left(q-1+\sqrt{(4n-2)q-3q^3+1}\right)$,
which asymptotically coincides with the bound in Theorem \ref{theorem2}
for $m_A={q\choose 2}$ and $m_{A,B}=q(n-q)$.

The next section contains the proofs of both results and some discussion.

\section{Proofs}

Since Theorem \ref{theorem1} relies on Theorem \ref{theorem2},
we start with the latter.

\begin{proof}[Proof of Theorem \ref{theorem2}]
For $m_{A,B}=0$, Stanley's result implies the desired bound.
Hence, we may assume that $m_{A,B}>0$, which implies $\lambda=\lambda(G)>0$.
Let $x=(x_u)_{u\in V(G)}$ be an eigenvector for the eigenvalue $\lambda$.
For every vertex $u$ of $G$, 
we have $\lambda x_u=\sum\limits_{v:v\in N_G(u)}x_v$
and applying this identity twice, we obtain
\begin{eqnarray}\label{e1}
\lambda^2 x_u
&=&\sum\limits_{v:v\in N_G(u)}\lambda x_v\nonumber
=\sum\limits_{v:v\in N_G(u)}\left(\sum\limits_{w:w\in N_G(v)}x_w\right)\nonumber\\
&=& x_ud_G(u)+\sum\limits_{v:v\in N_G(u)}\left(\sum\limits_{w:w\in N_G(v)\setminus \{ u\}}x_w\right);
\end{eqnarray}
this observation seems to originate from Favaron et al.~\cite{famasa}.

By the Perron-Frobenius Theorem and by normalizing the eigenvector $x$, 
we may assume that $x$ has no negative entry and that $\max\{ x_u:u\in V(G)\}=1$.
Let the vertex $u'$ be such that $x_{u'}=1$ and let $\alpha=\max\{ x_u:u\in B\}$.

If $\alpha=1$, then we may assume $u'\in B$ and 
applying (\ref{e1}) with $u=u'$ implies
\begin{eqnarray}
\lambda^2
&=& d_G(u')+\underbrace{\sum\limits_{v:v\in N_G(u')}\left(\sum\limits_{w:w\in N_G(v)\setminus \{ u'\}}x_w\right)}_{(*)}\nonumber\\
& \leq & d_G(u')+2m_A+(m_{A,B}-d_G(u'))\label{e2}\\
& = & 2m_A+m_{A,B},\label{e3}
\end{eqnarray}
where (\ref{e2}) follows because 
each of the $m_A$ edges $vw$ with $v,w\in A$ 
contributes at most $x_v+x_w\leq 2$ to $(*)$
and each of the $m_{A,B}-d_G(u')$ edges $vw$ with $v\in A$ and $w\in B\setminus \{ u'\}$ 
contributes at most $x_w\leq \alpha=1$ to $(*)$.

See Figure \ref{fig1} for an illustration.

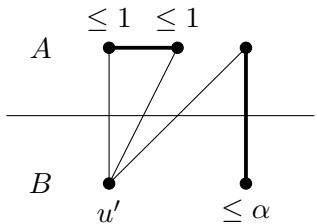
\begin{figure}[H]
   \centering
   \begin{tikzpicture} [scale=0.9]
       \tikzstyle{point}=[draw,circle,inner sep=0.cm, minimum size=1.5mm, fill=black]
       \tikzstyle{line1}=[line width=0.5mm]
       \tikzstyle{line2}=[thin]    

       \coordinate (a1) at (0,0) [label=left:] {};
       \coordinate (a2) at (1,0) [label=left:] {};
       \coordinate (a3) at (2,0) [label=left:] {};

       \node[point] at (0,0) [label=above:$\le 1$] {};
       \node[point] at (1,0) [label=above:$\le 1$] {};
       \node[point] at (2,0) [label=above:] {};
   
       \draw (-1.5,-1) -- (3,-1);

       \coordinate (b1) at (0,-2) [label=left:] {};
       \coordinate (b2) at (2,-2) [label=left:] {};

       \node[point] at (0,-2) [label=below:$u'$] {};
       \node[point] at (2,-2) [label=below:$\le \alpha$] {};

       \draw[line2] (b1) -- (a1);
       \draw[line2] (b1) -- (a2);
       \draw[line2] (b1) -- (a3);

       \draw[line1] (b2) -- (a3);
       \draw[line1] (a1) -- (a2);

       \node[align=center] at (-1,0) {$A$};
       \node[align=center] at (-1,-2) {$B$};

   \end{tikzpicture}
 \caption{Two edges incident with neighbors of $u'$ and their possible contributions to $(*)$.
 If one of the thin edges does not belong to $G$, the contribution is reduced accordingly.}
   \label{fig1}
\end{figure}

Since (\ref{e3}) is stronger than the stated bound, the proof is complete in this case.
Hence, we may assume that $\alpha<1$, which implies that $u'\in A$.

Let $u'$ have $d_A$ neighbors in $A$ and $d_{A,B}$ neighbors in $B$.
Applying (\ref{e1}) with $u=u'$ implies 
\begin{eqnarray}
\lambda^2
&=& d_G(u')+\underbrace{\sum\limits_{v:v\in N_G(u')}\left(\sum\limits_{w:w\in N_G(v)\setminus \{ u'\}}x_w\right)}_{(**)}\nonumber\\
& \leq & (d_A+d_{A,B})+2(m_A-d_A)+(1+\alpha)(m_{A,B}-d_{A,B})\label{e4}\\
& \leq & 2m_A+(1+\alpha)m_{A,B},\label{e5}
\end{eqnarray}
where (\ref{e4}) follows because 
each of the $m_A-d_A$ edges $vw$ with $v,w\in A\setminus \{ u'\}$ 
contributes at most $x_v+x_w\leq 2$ to $(**)$
and each of the $m_{A,B}-d_{A,B}$ edges $vw$ with $v\in A\setminus \{ u'\}$ and $w\in B$ 
contributes at most $x_v+x_w\leq 1+\alpha$ to $(**)$;
recall that $x$ has no negative entry, which implies $1\leq 1+\alpha$.

See Figure \ref{fig2} for an illustration.\\[-12mm]

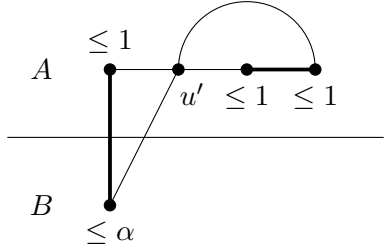
\begin{figure}[H]
   \centering
   \begin{tikzpicture} [scale=0.9]
       \tikzstyle{point}=[draw,circle,inner sep=0.cm, minimum size=1.5mm, fill=black]
       \tikzstyle{line1}=[line width=0.5mm]
       \tikzstyle{line2}=[thin]    

       \coordinate (a1) at (0,0) [label=left:] {};
       \coordinate (a2) at (1,0) [label=left:] {};
       \coordinate (a3) at (2,0) [label=left:] {};
       \coordinate (a4) at (3,0) [label=left:] {};

       \node[point] at (0,0) [label=above:$\le 1$] {};
       \node[point] at (1,0) [label=below right:] {};
       \node at (1.2,0.1) [label=below:$u'$] {};
       \node[point] at (2,0) [label=below:$\le 1$] {};
       \node[point] at (3,0) [label=below:$\le 1$] {};
   
       \draw (-1.5,-1) -- (4,-1);

       \coordinate (b1) at (0,-2) [label=left:] {};

       \node[point] at (0,-2) [label=below:$\le \alpha$] {};

       \draw[line2] (a1) -- (a2);
       \draw[line2] (a2) -- (a3);
       \draw[line2] (b1) -- (a2);
       \draw[line2] (a2) arc(180:0:1);

       \draw[line1] (b1) -- (a1);
       \draw[line1] (a3) -- (a4);

       \node[align=center] at (-1,0) {$A$};
       \node[align=center] at (-1,-2) {$B$};

   \end{tikzpicture}
   \caption{Two edges incident with neighbors of $u'$ and their possible contributions to $(**)$.}
   \label{fig2}
\end{figure}

If $\alpha=0$, then (\ref{e5}) is stronger than the stated bound
and the proof is complete in this case.
Hence, we may assume that $\alpha>0$.
Note that the inequality (\ref{e5}) is strict if $d_A>0$ or $\alpha d_{A,B}>0$,
that is, there is a tiny room for improvement.

Let $u''\in B$ be such that $x_{u''}=\alpha$.
Applying (\ref{e1}) with $u=u''$ implies 
\begin{eqnarray}
\lambda^2\alpha
&=& \alpha d_G(u'')+\sum\limits_{v:v\in N_G(u'')}\left(\sum\limits_{w:w\in N_G(v)\setminus \{ u''\}}x_w\right)\nonumber\\
& \leq & \alpha d_G(u'')+2m_A+\alpha(m_{A,B}-d_G(u''))\label{e6}\\
& = & 2m_A+\alpha m_{A,B},\label{e7}
\end{eqnarray}
where (\ref{e6}) follows similary as (\ref{e2}).
Since $\alpha>0$, the bound (\ref{e7}) implies 
\begin{eqnarray}
\lambda^2 & \leq & \frac{2}{\alpha}m_A+m_{A,B}.\label{e8}
\end{eqnarray}
Since the bound in (\ref{e5}) is increasing in $\alpha$ and 
the bound in (\ref{e8}) is decreasing in $\alpha$, 
we obtain that 
$\lambda^2\leq 2m_A+(1+\alpha^*)m_{A,B}$,
where $\alpha^*$ is chosen such that 
$2m_A+(1+\alpha^*)m_{A,B}=\frac{2}{\alpha^*}m_A+m_{A,B}$.
Solving this equation for $\alpha^*$ yields 
$\alpha^*=\sqrt{\left(\frac{m_A}{m_{A,B}}\right)^2+2\frac{m_A}{m_{A,B}}}-\frac{m_A}{m_{A,B}}\in [0,1]$.
Substituting this value in $\lambda^2\leq 2m_A+(1+\alpha^*)m_{A,B}$ yields
$$\lambda^2\leq m_A+m_{A,B}+\sqrt{m_A^2+2m_Am_{A,B}},$$ 
which completes the proof.
\end{proof}

\begin{proof}[Proof of Theorem \ref{theorem1}]
Let $\lambda=\lambda(G)$, $s=s(G)$, $d=\left\lceil\frac{2m}{n}\right\rceil$,
and $C=\{ u\in V(G):d_G(u)\geq d+1\}$.

We choose a set $E_0$ of edges of $G$ with both endpoints in $C$ such that
\begin{enumerate}[(i)]
\item $d_H(u)\geq d$ for every vertex $u$ in $C$ and the graph $H=G-E_0=(V(G),E(G)\setminus E_0)$,
\item subject to condition (i), the number $m_0=|E_0|$ of edges in $E_0$ is as large as possible, and
\item subject to conditions (i) and (ii), the expression 
$$\sum\limits_{u\in C}\max\{ d_H(u)-(d+1),0\}$$ is as small as possible.
\end{enumerate}
Let $C'=\{ u\in C:d_H(u)=d\}$.
Let $C''$ be the set of isolated vertices of the graph $\left(C',E_0\cap {C'\choose 2}\right)$.
Let $A=C\setminus C''$ and $B=V(G)\setminus A$.

See Figure \ref{fig3} for an illustration.

\begin{figure}[H]
   \centering
   \begin{tikzpicture} [scale=0.8]
       \tikzstyle{point}=[draw,circle,inner sep=0.cm, minimum size=1.5mm, fill=black]
       \tikzstyle{line1}=[line width=0.4mm]
       \tikzstyle{line2}=[line width=0.4mm]    

       \coordinate (x1) at (1,2) [label=left:] {};
       \coordinate (x2) at (2,2) [label=left:] {};
       \coordinate (x3) at (4,2) [label=left:] {};

       \foreach \i in {1,2,4}{
           \node[point] at (\i,2) [label=below:] {};
       }

       \coordinate (a1) at (1,0) [label=left:] {};
       \coordinate (a2) at (3,0) [label=left:] {};
       \coordinate (a3) at (6,0) [label=left:] {};

       \node[point] at (1,0) [label=above:] {};
       \node[point] at (3,0) [label=below right:] {};
       \node[point] at (6,0) [label=below:] {};
   
       \draw (-1,-1) -- (7,-1);
       \draw (-1,1) -- (7,1);
       \draw (-1,-3) -- (7,-3);
       \draw (7,3) -- (7,-3);
       \draw (-1,-3) -- (-1,3);
       \draw (3,-1) -- (3,-3);

       \coordinate (b1) at (0,-2) [label=left:] {};
       \coordinate (b2) at (1,-2) [label=left:] {};
       \coordinate (b3) at (2,-2) [label=left:] {};
       \coordinate (b4) at (4,-2) [label=left:] {};
       \coordinate (b5) at (5,-2) [label=left:] {};
       \coordinate (b6) at (6,-2) [label=left:] {};

    \foreach \i in {0,1,2,4,5,6}{
           \node[point] at (\i,-2) [label=below:] {};
       }

       \draw[line1] (b1) -- (a1);
       \draw[line1] (b2) -- (a1);
       \draw[line1] (b3) -- (a1);
       \draw[line1] (b3) -- (a2);
       \draw[line1] (b4)-- (b5) -- (b6) -- (a3);
       
       \draw[line1] (a1) -- (x1) -- (x2) -- (a2) -- (x3) -- (b4);

       \node[align=center] at (8,-2) {$d$};
       \node[align=center] at (8,0) {$d+1$};
       \node[align=center] at (8,2) {$\ge d+2$};
       \node[align=center] at (2,-4.4){$B=(V(G) \setminus C) \cup C''$};
       \node[align=center] at (1.0,2.7) {$A=C\setminus C''$};
       \node[align=center] at (0,-2.5) {${C''}$};

       \draw[line2] (-1.15,-5) -- (-1.15,-1.15) -- (2.85,-1.15) -- (2.85, -3.15) -- (7,-3.15) -- (7,-5);
        \draw[thick,black,decorate,decoration={brace,amplitude=5}] (9,3) -- (9,-3) node[midway, below,yshift=0]{};    
       \node[align=left] at (13,0) {
The vertices in the set $C$\\
grouped according their \\
degrees in $H=G-E_0$};
    
       \draw[line1] (-1,-1) -- (3,-1) -- (3,-3) -- (-1,-3) -- (-1,-1) -- (3,-1);
    
   \end{tikzpicture}
   \caption{The partition of the vertex set of $G$ into $A$ and $B$. The edges shown within $C$ are the edges in $E_0$
   that are removed from $G$ to obtain $H$.
For the vertices in $C$, we consider their degrees $d$, $d+1$, and $\geq d+2$ in $H$.}
   \label{fig3}
\end{figure}
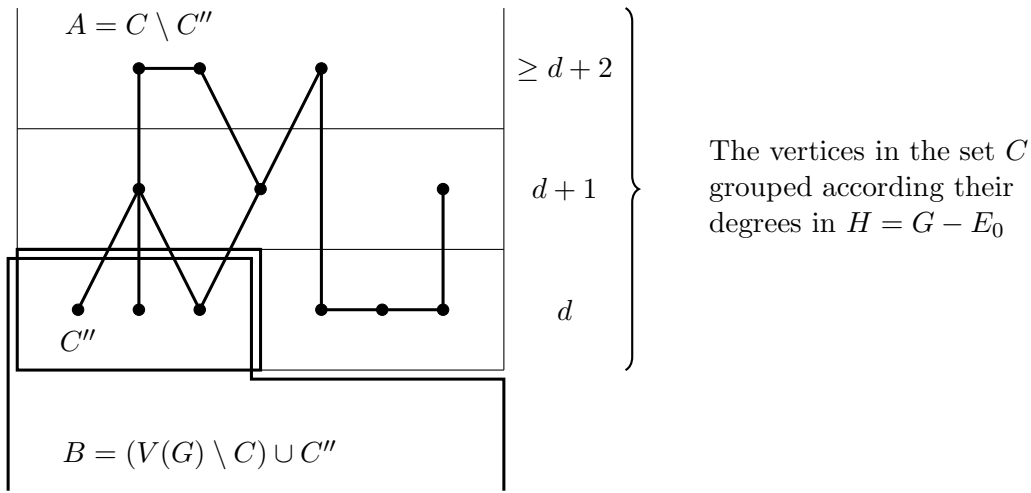

By (ii) in the choice of $E_0$, the set $C\setminus C'=\{ u\in C:d_H(u)\geq d+1\}$ is independent in $H$.
If $uv\in E(H)$ with $d_H(u)\geq d+2$ and $v\in C'\setminus C''$,
then $E_0$ contains an edge $vw$ with $w\in C'\setminus C''$ and
$E_0'=(E_0\setminus \{ vw\})\cup \{ uv\}$ yields a contradiction to the condition (iii) in the choice of $E_0$.
Hence, in the graph $H$, the vertices in $\{ u\in C:d_H(u)\geq d+2\}$ have all their neighbors in $B$.
Let $E_A=E_0\cap {A\choose 2}$.
Note that all edges in $E_0\setminus E_A$ are between $A$ and $B$.
Let $E_{A,B}$ arise from $E_0\setminus E_A$ by adding, for every vertex $u\in C$ with $d_H(u)\geq d+2$,
exactly $d_H(u)-(d+1)$ edges incident with $u$.
By construction,
all edges in $E_A$ have both their endpoints in $A\subseteq C$ and
every edge in $E_{A,B}$ connects a vertex from $A$ to a vertex from $B$.
Furthermore, the graph $G'=G-(E_A\cup E_{A,B})$ has maximum degree at most $d+1$
and $d_{G'}(u)\geq d$ holds for every vertex $u\in A$.
Let $m_A=|E_A|$, $m_{A,B}=|E_{A,B}|$, and $G''=(V(G),E_A\cup E_{A,B})$.

Since $\sum\limits_{u\in V(G)}\left(d_G(u)-\frac{2m}{n}\right)=0,$ we have 
\begin{eqnarray}
2m_A+m_{A,B}
\leq \sum\limits_{u\in A}(d_G(u)-d)
\leq \sum\limits_{u\in C}(d_G(u)-d)
\leq \sum\limits_{u\in C}\left(d_G(u)-\frac{2m}{n}\right)
\leq \frac{s}{2}.\label{e9}
\end{eqnarray}
Since $G$ is the edge-disjoint union of the graphs $G'$ and $G''$, 
we obtain using the maximum degree bound for $G'$ and
Theorem \ref{theorem2} for $G''$ that

\begin{eqnarray}
\lambda & \leq & \lambda(G')+\lambda(G'')\label{e9b}\\
& \leq & d+1+\sqrt{m_A+m_{A,B}+\sqrt{m_A^2+2m_Am_{A,B}}}.\label{e10}
\end{eqnarray}
Since (\ref{e10}) is increasing in $m_{A,B}$, it follows using (\ref{e9}) that
\begin{eqnarray}
\lambda 
& \leq & 
d+1+\max\left\{ \sqrt{x+y+\sqrt{x^2+2xy}}:x,y\geq 0\mbox{ and }2x+y=\frac{s}{2}\right\}\nonumber\\
& \leq & 
d+1+\max\left\{ \sqrt{\frac{s}{2}-x+\sqrt{x(s-3x)}}:0\leq x\leq \frac{s}{4}\right\}.\label{e11}
\end{eqnarray}
A simple calculation shows that $x=\frac{s}{12}$ solves the maximization problem in (\ref{e11}) and we obtain
\begin{eqnarray*}
\lambda 
& \leq & 
d+1+\sqrt{\frac{s}{2}-\frac{s}{12}+\sqrt{\frac{s}{12}\left(s-3\frac{s}{12}\right)}}=d+1+\sqrt{\frac{2s}{3}}.
\end{eqnarray*}
At this point, we have $\lambda-\frac{2m}{n}\leq \lambda-d+1\leq \sqrt{\frac{2s}{3}}+2$.
Now, Nikiforov's blow-up argument (cf. proof of Theorem 8 in \cite{ni}), 
replacing every vertex of $G$ by an independent set of order $t$
and letting $t$ tend to infinity,
implies $\lambda-\frac{2m}{n}\leq \sqrt{\frac{2s}{3}}$,
which completes the proof.
\end{proof}
We believe that the estimate (\ref{e9b}) is the crucial point within the above proof
that is too weak to establish Nikiforov's conjecture.

\end{document}